\documentclass{emsprocart}

\usepackage[unicode=true,pdfborder={0 0 0},colorlinks=true,pdfborderstyle={},linkcolor=black,citecolor=black,urlcolor=blue]{hyperref}
\usepackage{calrsfs}
\usepackage{graphicx}
\usepackage{color}
\usepackage{auto-pst-pdf}

\contact[gadyk@weizmann.ac.il]{Gady Kozma, Department of Mathematics,
  Weizmann, Rehovot 76100, Israel.}

\theoremstyle{plain}
\newtheorem{thm}{Theorem}

\DeclareMathOperator{\len}{len}
\DeclareMathOperator{\dist}{dist}

\title{Reinforced random walk}

\author[Gady Kozma]{Gady Kozma\thanks{Research partially
    supported by the Israel Science Foundation}}

\begin{document}

\begin{abstract}
We show that linearly reinforced random walk has a recurrent phase in
every graph. The proof does not use the magic formula. 
\end{abstract}



\maketitle

\section{Introduction}
Let $f$ be some function from the positive integers $\{0,1,2,\dotsc\}$
into $(0,\infty)$, and let $G$ be a graph and $x_{0}$ some vertex
of it. Then we define \emph{reinforced random walk} on $G$ starting
at $x_{0}$ with reinforcement function $f$ as follows. It is a walk
on $G$ i.e.\ a random sequence of vertices $x_{0},x_{1},\dotsc$
with $x_{i+1}$ a neighbour of $x_{i}$; and the transition probabilities
depend only on the past. Specifically, for an edge $e$ and time $t$
we define the number of traversals $N(e,t)$ using
\[
N(e,t)=\#\{s\mbox{ such that }1\le s\le t,\:(x_{s-1},x_{s})=e\}
\]
where $(x_{s-1},x_{s})$ is an unordered edge, so $N$ counts traversals
of $e$ in both directions. Now apply $f$ and normalize to get probabilities.
Totally, the definition is
\[
\mathbb{P}(x_{t+1}=x\,|\, x_{0},\dotsc,x_{t})=\frac{f(N((x,x_{t}),t))}{\sum\limits _{\smash{y\sim x_{t}}}f(N((y,x_{t}),t)}
\]
where $\sim$ is the neighbourhood relation in our graph. This finishes
the definition of the process. This is well defined assuming all degrees
of all vertices are finite, and we assume this from now on on all
our graphs. There are also versions where reinforcement is applied
to vertices or to directed edges, but for simplicity we concentrate
for now on edge reinforced walks.

It comes as no surprise that different $f$ give different processes,
sometimes dramatically so. To make things more concrete let us discuss
a few examples.
\begin{enumerate}
\item Once-reinforced random walk is the process defined by 
\[
f(n)=\begin{cases}
1 & n=0\\
2 & n\ge1
\end{cases}.
\]
In other words, edges which have already been traversed are given
double weight when compared to unknown edges, but the number of
traversals does not matter.
\item Linearly reinforced random walk is the process defined by $f(n)=1+n$.
In other words, whenever you traverse an edge, you increase its weight
by $1$. 
\item Strongly reinforced random walk, where $f$ grows faster than linear.
For example, one may take $f(n)=1+n^{2}$.
\end{enumerate}
Perhaps surprisingly, the order of difficulty seems to be decreasing.
Once-reinforced random walk is the least understood, while the superlinear
case is the simplest.

Let us start with the superlinear case. It turns out that in this
case the walk gets stuck on a single edge, going back and forth endlessly.
To understand why, let us prove the following weaker claim:
\begin{thm}
Let $G$ be an arbitrary graph, $x_{0}$ a vertex of $G$ and $e$
an edge going out of $x$. Examine reinforced random walk on $G$
starting from $x_{0}$ with reinforcement function $1+n^{2}$. Then
with positive probability the process goes back and forth over $e$
endlessly i.e.\ for all $t$, $x_{2t}=x_{0}$ and $x_{2t+1}$ is
the other vertex of $e$.\end{thm}
\begin{proof}
Denote the other vertex of $e$ by $y$. Let $E_{s}$ be the event
that $x_{2t}=x_{0}$ and $x_{2t+1}=y$ for all $2t\le s$ and $2t+1\le s$,
respectively. Let us calculate the probability of $E_{s+1}$ given
$E_{s}$. The numbers of traversals are $N(e,s)=s$ and $N(f,s)=0$
for any edge $f\ne e$. Hence the probability is 
\[
\frac{f(s)}{f(s)+(\deg x_{s}-1)f(0)}=\frac{1+s^{2}}{1+s^{2}+(\deg x_{s}-1)}\ge1-\frac{C}{s^{2}}.
\]
for some constant $C$ that depends only on the degrees of $x_{0}$
and $y$. Therefore the probability of $E_{s}$ is 
\[
\mathbb{P}(E_{s})=\prod_{u=1}^{s}\mathbb{P}(E_{u}|E_{u-1})\ge\prod_{u=1}^{s}\left(1-O\left(\frac{1}{u^{2}}\right)\right)\ge c
\]
for some positive $c$. This proves the claim.
\end{proof}
There is no zero-one law for the behaviour of processes like reinforced
random walk, so one cannot easily conclude from the theorem above
that some event of interest happens with probability 1. Nevertheless,
this has been proved with almost complete generality. 
\begin{thm}
Let $G$ be a graph and $f$ a function satisfying 
\[
\sum_{n=0}^{\infty}\frac{1}{f(n)}<\infty.
\]
Then reinforced random walk on $G$ with reinforcement function $f$
eventually gets stuck on a single edge with probability $1$, under
either of the following conditions
\begin{enumerate}
\item $f$ is increasing, or
\item $G$ contains no cycle of odd order.
\end{enumerate}
\end{thm}
See Limic \cite{L03} and Limic and Tarr\`es \cite{LT07,LT08}. It
is conjectured that neither condition is necessary, but clearly, the
result is quite complete as is.

\section{Phase transitions}

Let us now consider once-reinforced random walk, but let us add a
parameter $a$ and examine the reinforcement function 
\[
f(n)=\begin{cases}
1 & n=0\\
1+a & n\ge1.
\end{cases}
\]
An exciting conjecture due to Beffara \& Sidoravicius is that, say for the graph
$\mathbb{Z}^{d}$, $d\ge 3$, there is a \emph{phase transition} in
$a$. For $a$ very small the walk is essentially simple random walk.
For $a$ very large the walk looks like a man trapped in a balloon:
it constantly jumps against the boundaries of the balloon, occasionally
inflating it, but the balloon keeps its round shape, as it inflates.
The ``balloon'', of course, is the set of visited sites. In particular
for $d\ge 3$ there should be a phase transition between recurrence and
transience. Let us stress
that $a$ does not depend on $t$. It is not difficult to establish
results if one first fixes $t$ and then makes $a$ very small or
very large. In this case for $a\ll\frac{1}{t}$ one can directly check
that the process is identical to simple random walk (with high probability),
while for $a\gg t$ the process gets stuck on its first edge. But
the conjecture is that for a fixed $a$ there is a marked difference
in behaviour as $t\to\infty$. Some simulation results can be seen
in figure \ref{fig:Once}. 
\newdimen\figwid
\figwid=0.48\columnwidth
\begin{figure}
\centering\begin{minipage}[t]{\figwid}%
\includegraphics[width=1\textwidth]{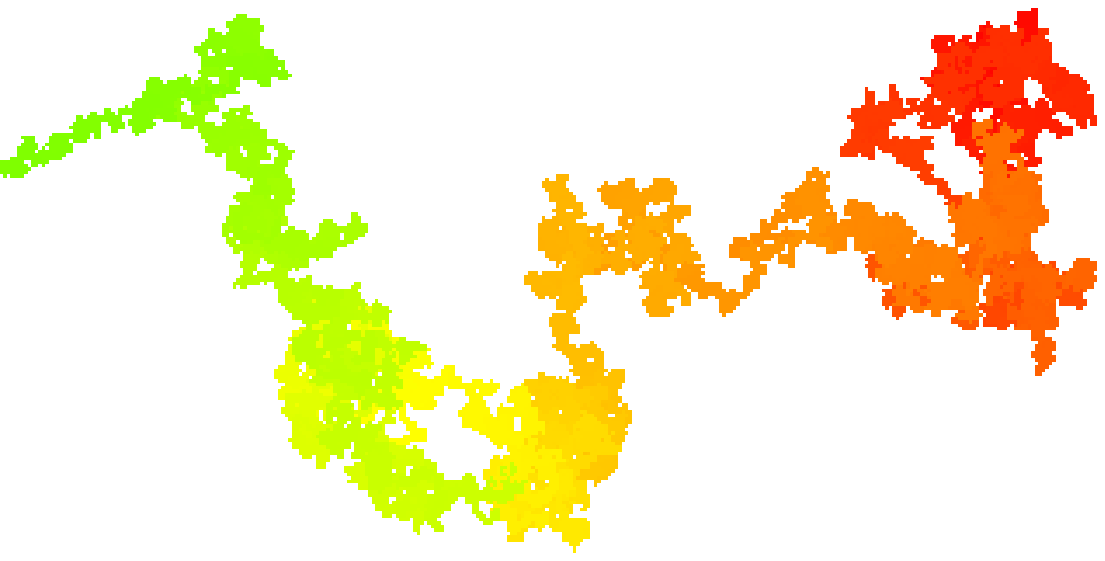}

\begin{center}
$1+a=2$
\par\end{center}%
\end{minipage}%
\begin{minipage}[t]{\figwid}%
\includegraphics[width=1\textwidth]{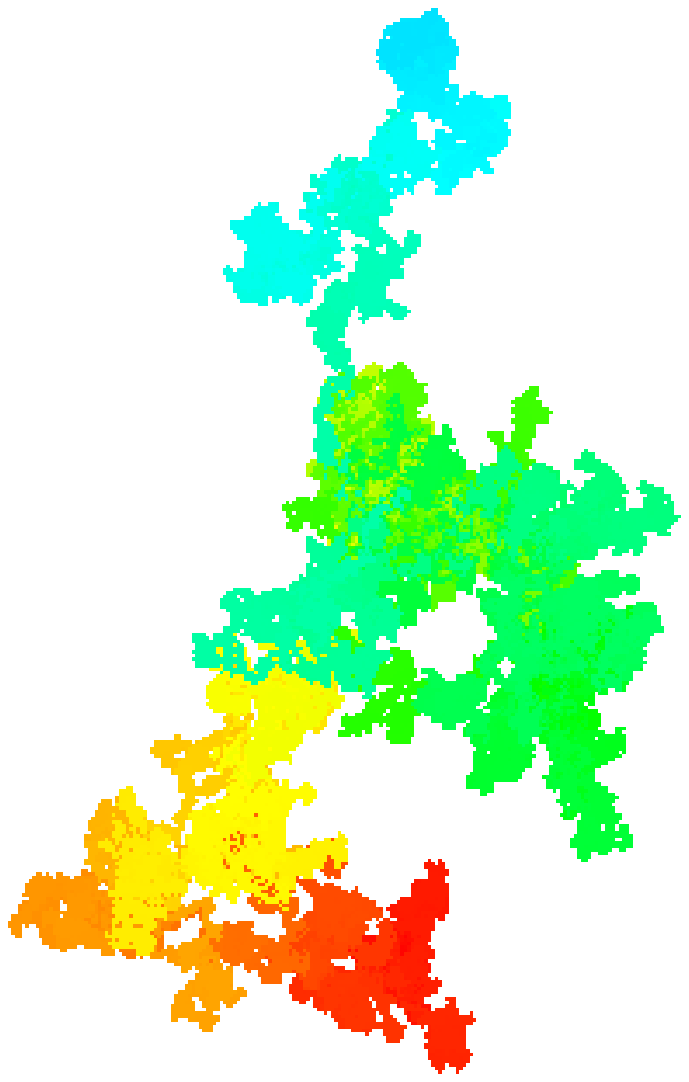}

\begin{center}
$1+a=3$
\par\end{center}%
\end{minipage}

\begin{minipage}[t]{\figwid}%
\includegraphics[width=1\textwidth]{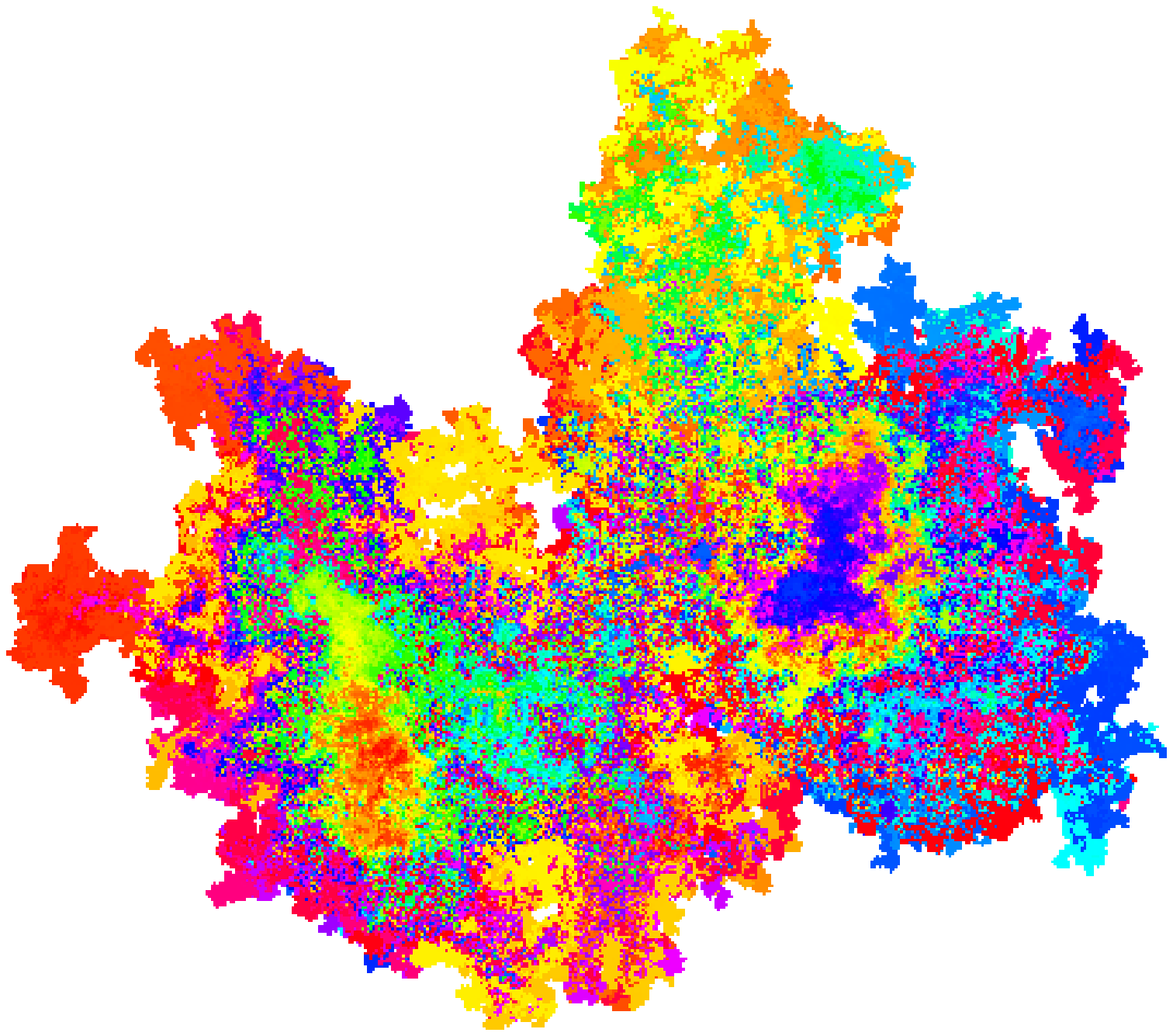}

\begin{center}
$1+a=4$
\par\end{center}%
\end{minipage}%
\begin{minipage}[t]{\figwid}%
\includegraphics[width=1\textwidth]{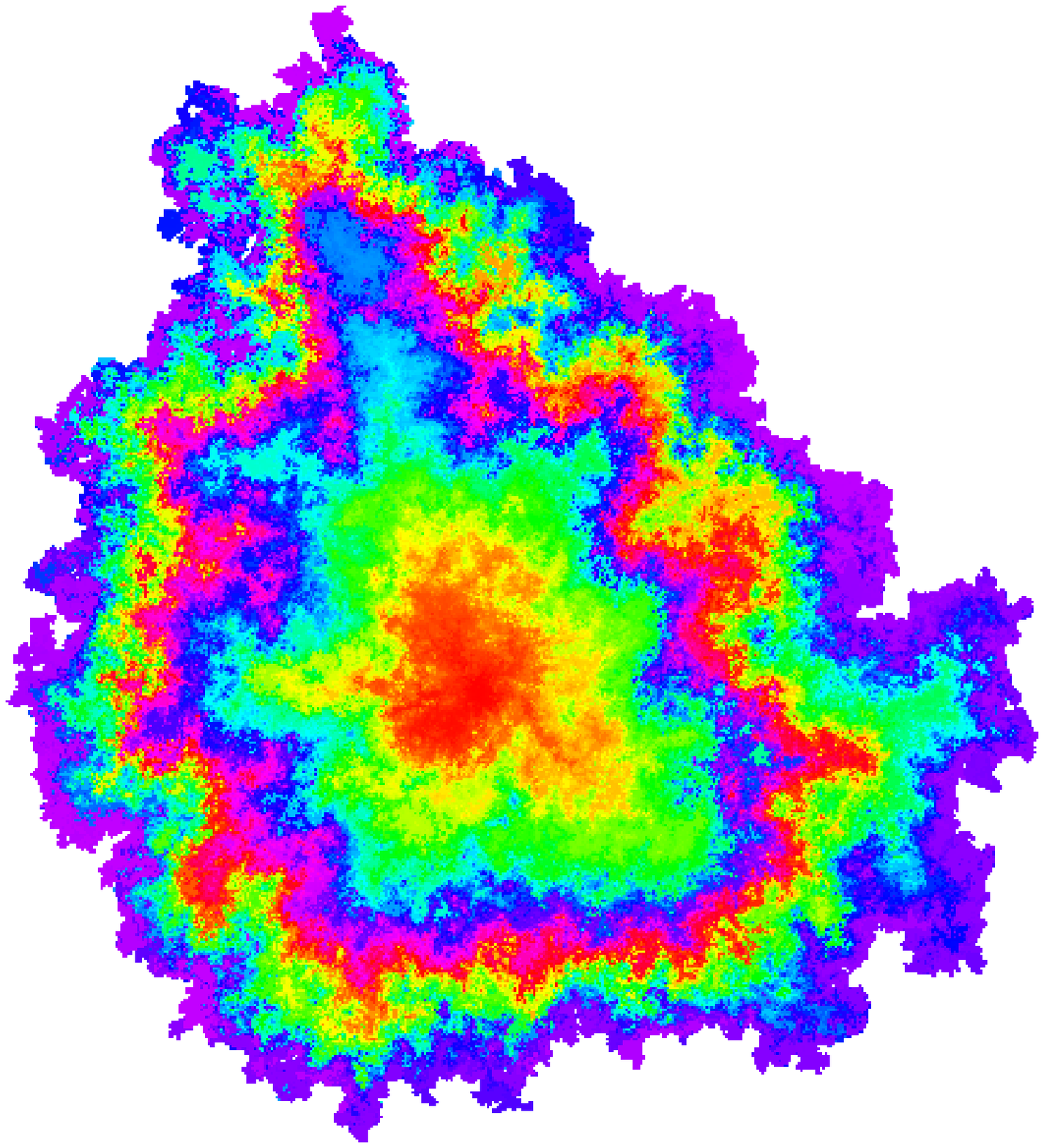}

\begin{center}
$1+a=5$
\par\end{center}%
\end{minipage}

\begin{minipage}[t]{\figwid}%
\includegraphics[width=1\textwidth]{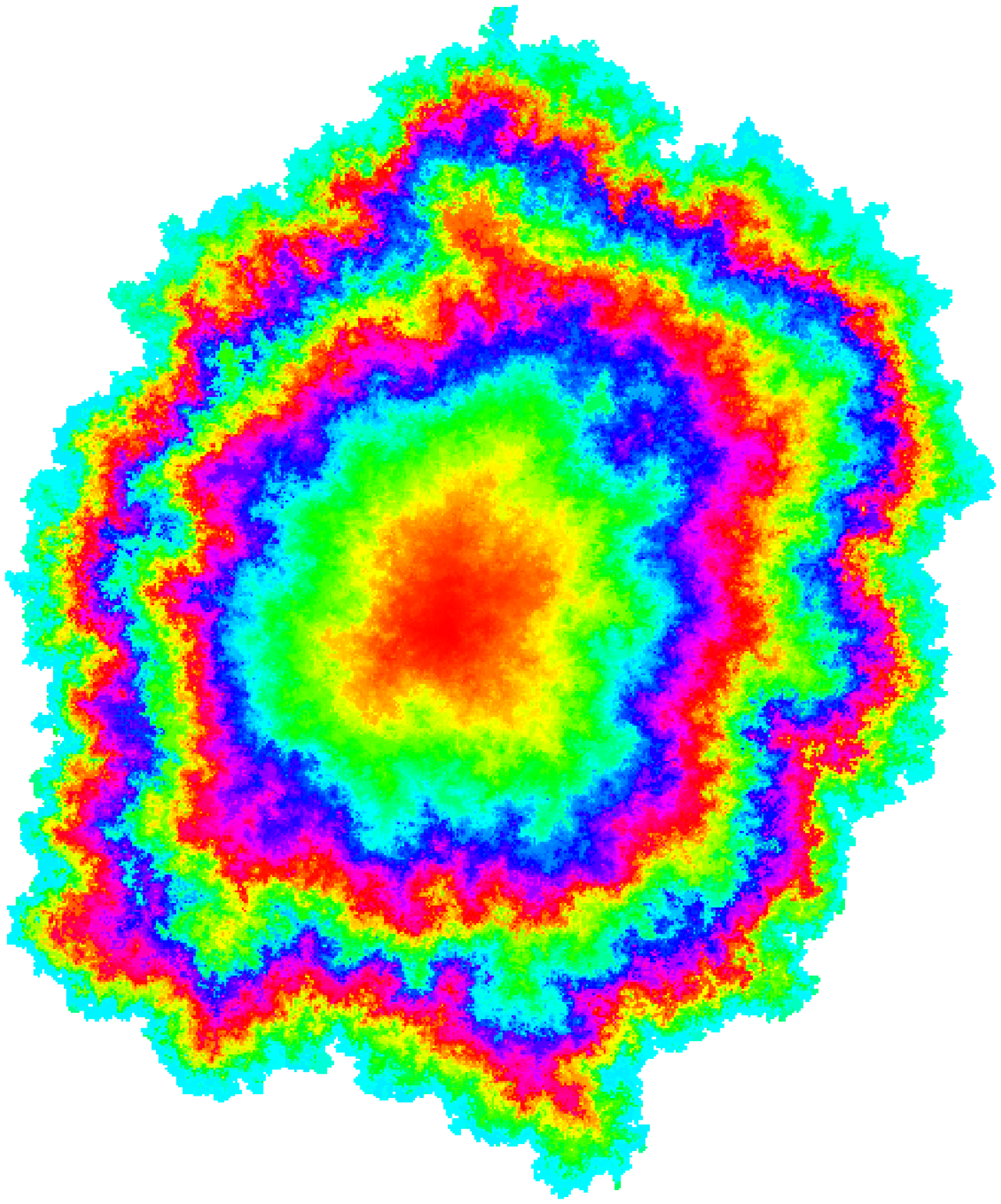}

\begin{center}
$1+a=6$
\par\end{center}%
\end{minipage}%
\begin{minipage}[t]{\figwid}%
\includegraphics[width=1\textwidth]{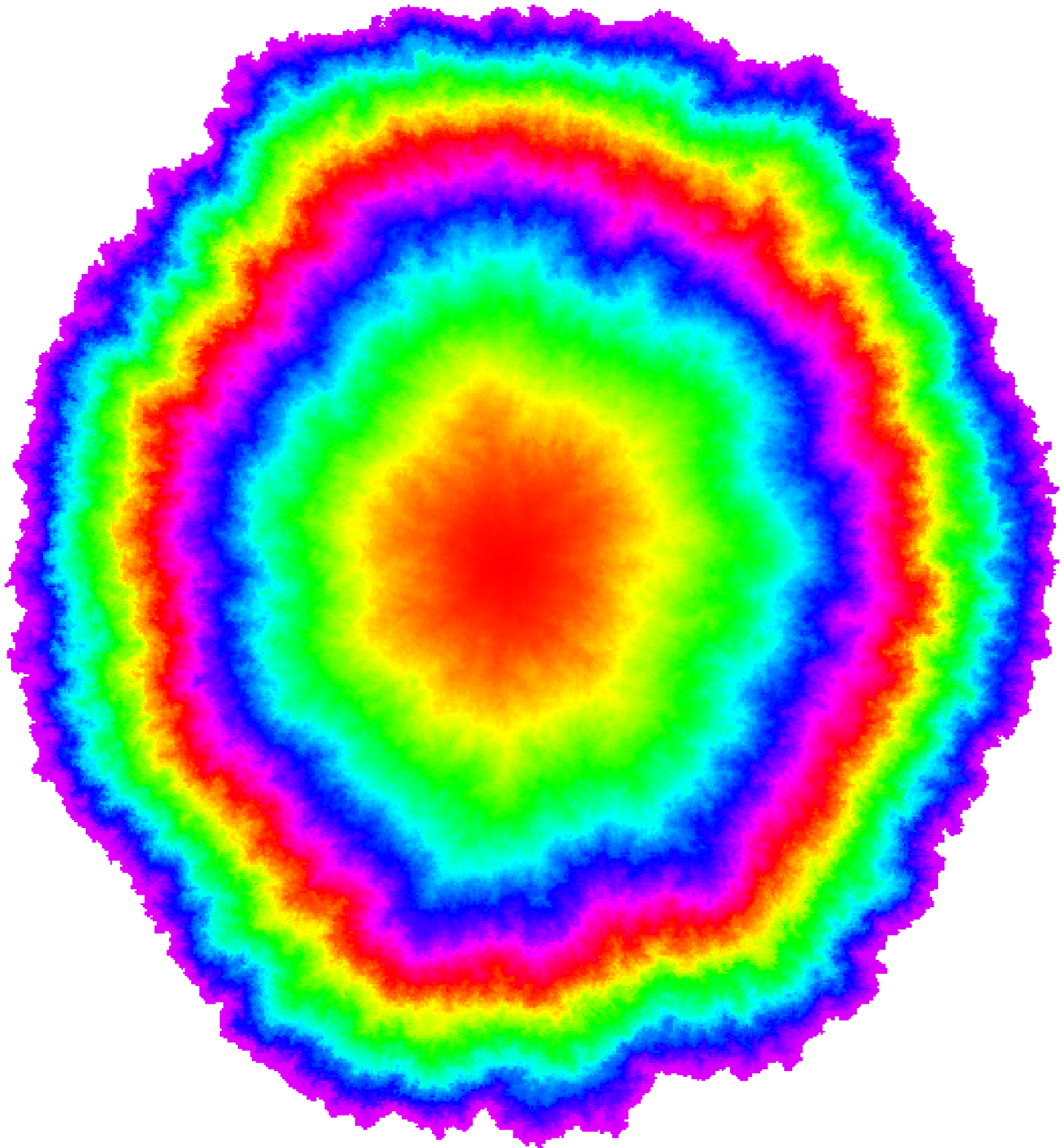}

\begin{center}
$1+a=10$
\par\end{center}%
\end{minipage}

\caption{\label{fig:Once}Once-reinforced random walk}
\end{figure}
 Each simulation was done on a $5000\times5000$ grid and stopped
when the process reached the boundaries. The 2-dimensional case  was chosen
since it is easiest to depict, but one should probably caution
that in fact it is not clear what is the true behaviour in 2 dimensions,
even on a conjectural level. Colour depicts the number
of reinforced edge. This means that the second edge which is reinforced
is given colour $2$, not the second edge which is traversed (colour
change rate is not identical in the 6 images). The fractal-like colour
patterns seen in $1+a=4$ are indicative of the process leaving holes
all around and then going back to fill them at a macroscopically later
time. Presumably $4$ is quite close to the point of criticality where
the processes of roughening and smoothing balance exactly. 

The case of large $a$ in this conjecture is related to a series of
similar models. Internal diffusion-limited aggregation is a process
where the walker, upon reaching a new vertex, is transported to the
starting point $x_{0}$. For this model it is known since the early
90s that the set of visited vertices is spherical \cite{LBG92,L95},
with impressive improvements in the precision of the results achieved
lately \cite{JLS12,JLS11,AG11}. A variation where, when the walker,
upon reaching a new vertex is transported to a random vertex in the
existing cluster also gives a ball \cite{BDCKL}. Getting closer to
once-reinforced random walk we may examine \emph{excited to the center}.
In this model the walker, upon getting to a new vertex gets a small
drift towards the center. Less is known about this problem, basically
only that the process is recurrent in any dimension, regardless of
the strength of the drift {[}unpublished, 2006{]}. Once-reinforced
has all the difficulties of excited to the center and more: there
is no preferred direction and there is a phase transition. Hence it
seems quite challenging at this point. Existing rigorous results on
the model are on a tree \cite{DKL02}, and it is not difficult to
solve the one-dimensional case, but neither case exhibits a phase
transition, so for now this phenomenon is beyond reach.

Let us close this comparison by returning to the simulation results
above. As can be seen, the shapes for $1+a=5$ or $6$ are not that
round. This is unlike internal diffusion-limited aggregation where
even quite small simulations are very close to a ball. A similar phenomenon
can be seen in simulations of excited to the center done by Yadin.
This is simply indicative that the boundary smoothing mechanism is
far less efficient in these models.

\section{P\'olya's urn}

We now move to the main topic of this talk, linearly reinforced random
walk, LRRW for short. Naively one might conjecture that, since $\sum\frac{1}{n}$
diverges so slowly, maybe the behaviour of LRRW is not that different
from the behaviour of strongly reinforced random walk. Let us start
with a simple example that shows that this is not at all the case.
The example will be a graph with $3$ vertices connected in a line,
and $x_{0}$ being the middle vertex. The reinforcement function will
be $2+n$. With this weight, the process is equivalent to the classic
P\'olya urn. Recall that in the P\'olya urn model one repeatedly
takes out a ball, and then returns two balls of the same colour as
the ball you took out. Looking at the weights of the LRRW at time
$2t$ (i.e.\ the values of $f(N(e,2t))$), and dividing them by $2$,
one gets a P\'olya urn. Let us recall P\'olya's theorem \cite{P30}\newdimen\wide
\newdimen\high
\wide=10.4cm
\high=2.2cm
\parshape 8 0pt \hsize 0pt \hsize 0pt \hsize 0pt \hsize 0pt \hsize 0pt \wide 0pt \wide 0pt \hsize
\vadjust{\kern -\high \vtop to \high{\hbox to \hsize{\hss\includegraphics{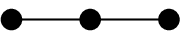}}}}
\begin{thm}
\label{thm:polya}At time $t$ (i.e.\ when the urn contains $t+1$
balls), the probability that it contains $k$ black balls is exactly
$\frac{1}{t}$, for any value of $k$ between $1$ and $t$.\end{thm}
\begin{proof}
Assume inductively that the claim has been proved for $t-1$. Now,
to get $k$ black balls at time $t$ you need one of the following:
\begin{enumerate}
\item Have $k-1$ black balls at time $t-1$ and take a black ball out of
the bin; or
\item Have $k$ black balls at time $t-1$ and take a white ball out of
the bin.
\end{enumerate}
The probability of (1) is 
\[
\frac{1}{t-1}\cdot\frac{k-1}{t}
\]
where the term $1/(t-1)$ is our inductive assumption, and the term
$(k-1)/t$ is the probability to pick a black ball when you have $k-1$
black balls out of a total of $t$ balls. This formula also holds
for $k=1$, though for a different reason: here the probability is
zero since $0$ black balls are not allowed, and the formula indeed
gives $0$. A similar calculation shows that the probability of (2)
is 
\[
\frac{1}{t-1}\cdot\frac{t-k}{t}.
\]
Summing (1) and (2) gives
\[
\frac{1}{t-1}\cdot\left(\frac{k-1}{t}+\frac{t-k}{t}\right)=\frac{1}{t}
\]
and the induction is complete.
\end{proof}
Thus we see that the picture is very different from that of strongly
reinforced random walk. Looking at strongly reinforced random
walk at some very large time would give that with probability very
close to $1$, the walk visited one of the edges a very small number
of times, and the other edge all the rest (from symmetry each one gets
probability close to $\frac{1}{2}$). Here this is not the case, typically
each edge is visited a number of times proportional to $t$. On the
other hand, this is also very different from what we expect for, say,
once-reinforced random walk. It is not difficult to analyze this explicitly
and see that once-reinforced would have the number of visits to each
edge at time $t$ being $\frac{1}{2}t$, with the random errors being
of order $\sqrt{t}$. So we see that the linear case is quite special.
\parshape 0

The analogy to P\'olya's urn works in a number of other cases. Examine
the graph $\mathbb{Z}$. Examine a single vertex. Since there are
no cycles, you know that when you leave your vertex through some edge
$e$, you are bound to return through $e$, if you return at all.
The other vertices can only effect you through preventing return,
not in any other way. Thus you can couple LRRW on $\mathbb{Z}$ to
a sequence of i.i.d.\ P\'olya urns. The same holds for any tree,
except that you get multicoloured P\'olya urns. But these are also
well understood. Using these ideas, Pemantle \cite{P88} was able
to analyze LRRW on regular trees. He showed that there is a phase
transition in the initial weight $a$
\begin{thm}
\label{thm:pmntl}Let $G$ be a regular tree of degree $\ge3$, and
let $a+n$ be the reinforcement function. Then there exists an $a_{c}$
such that for $a<a_{c}$ LRRW is (strongly) recurrent, while for $a>a_{c}$
it is transient.
\end{thm}
Here strong recurrence means that the expected time to return to your
starting point is finite. This is significantly stronger than usual
recurrence: infinite graphs (like $\mathbb{Z}$ or $\mathbb{Z}^{2}$)
are recurrent but not strongly so (for $\mathbb{Z}$ this is the well-known
``gambler's ruin'' theorem). In a way it means that the process
behaves like random walk on a \emph{finite} graph. Again, compare
to the case of once-reinforced random walk on a tree which is transient
for any value of the parameter \cite{DKL02}.

Theorem \ref{thm:pmntl} shows that importance of the parameter $a$,
and from now on we will simply say ``LRRW with $a$'' rather than
the cumbersome ``LRRW with reinforcement function $a+n$''. Note that
the parameter $a$ has the opposite effect compared to once-reinforced
random walk: there small $a$ gave random-walk-like behaviour, while
here it is the large $a$ that give this.

Another case where the theory of urns is applicable is \emph{directed}
reinforced random walk. One modifies the definition of $N$ to count
directed traversals i.e.
\[
N^{\textrm{dir}}(x,y,t)=\#\{1\le s\le t:x_{s-1}=x,x_{s}=y\}
\]
and then again define $x_{s}$ by taking $a+N^{\textrm{dir}}$ and
normalizing it. Here it is not important what graph one looks at,
each vertex can be isolated and treated as a multicoloured urn independent
of all others. Nevertheless, this urn argument only reduces the model
to a kind of \emph{random walk in random environment}, a notorious
problem in its own right. This was tackled on the graphs $\mathbb{Z}^{d}$
by Enriquez and Sabot \cite{ES03,S11}. Interestingly, there is no
phase transition when $d\ge3$, the walk is transient for any choice
of $a$.

We will return to P\'olya's urn later, but for now let us go back
to general graphs.

\section{Partial exchangeability}

Another way to understand why LRRW is special is that it has a property
called \emph{partial exchangeability}. This means that for a given
path, its probability depends on how many times each edge is visited,
but not on the order in which these visits happen. Figure \ref{fig:Partial-exchangeability}
\begin{figure}
\centering\input{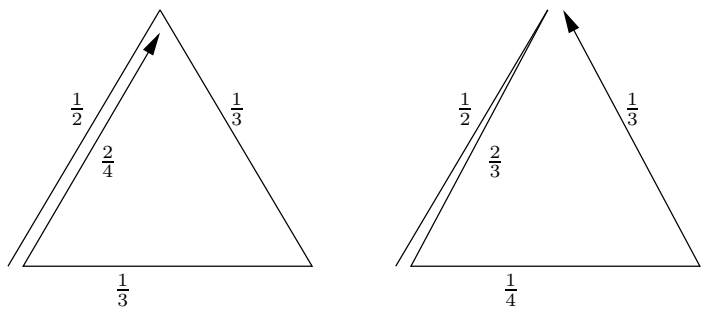}

\caption{\label{fig:Partial-exchangeability}Partial exchangeability}
\end{figure}
 shows two such paths in a triangle. Both paths have length 4, and
both cross one edge twice and the other two edges once, but are different.
The numbers in the figure are the probability of each traversal for
LRRW with $a=1$. The reader can verify that the numbers are not the
same, but their product is
\[
\frac{1}{2}\cdot\frac{1}{3}\cdot\frac{1}{3}\cdot\frac{2}{4}=\frac{1}{2}\cdot\frac{2}{3}\cdot\frac{1}{4}\cdot\frac{1}{3}
\]
Indeed the numerators on the left are a permutation of the numerators
on the right; and ditto for the denominators. The general proof is
the same:
\begin{thm}
LRRW is partially exchangeable.\end{thm}
\begin{proof}
We may assume without loss of generality that the reinforcement function
is $n+a$ for some $a>0$. Let $\gamma$ be a path in the graph $G$.
Then the probability $p(\gamma)$ that $\gamma=(\gamma_{0},\dotsc,\gamma_{l})$
is the beginning of the LRRW path is
\[
p(\gamma)=\prod_{i=0}^{l-1}\frac{a+N_{\gamma}((\gamma_{i},\gamma_{i+1}),i)}{\sum\limits _{\smash{y\sim\gamma_{i}}}a+N_{\gamma}((\gamma_{i},y),i)}
\]
where $N_{\gamma}$ is of course the analog of $N$ above for the
path $\gamma$ i.e.\ $N_{\gamma}(e,t)=\#\{1\le s\le t:(\gamma_{s-1},\gamma_{s})=e\}$.
Now, if $N_{\gamma}(e,l)=q$ for some $q$, then the numerators contain
the terms $a$, $a+1,\dotsc,a+q$ at the $q$ places where $e$ is
an edge of $\gamma$. This is not specific to linear reinforcement.
Now assume that the vertex $v\ne x_{0}$ appears $r$ times in the
path $\gamma_{0},\dotsc,\gamma_{l-1}$ (note that we do not count
$\gamma_{l}$).  Then the denominators contain $a\deg v+1$, $a\deg v+3,\dotsc,a\deg v+2r-1$
at the $r$ places where $\gamma_{i}=v$. For $v=x_{0}$ you get a
similar picture, except the denominators contain $a\deg x_{0},\dotsc,a\deg x_{0}+2r-2$.
In other words it does not matter at which order did $\gamma$ traverse
the edges coming out of $v$, since each such traversal increases
the total weight of the edges surrounding $v$ by $1$ (and another
$1$ is added when returning to the vertex). This is the point which
is special to linear reinforcement. 

Hence, if you know $N_{\gamma}(e,l)$ for all $e$ you can reconstruct
both numerators and denominators up to a permutation (this requires
the observation that you can identify the last vertex of $\gamma$,
since it is the only one with odd total number of appearances, or
$x_{0}$ if no vertex has this property). Thus you can calculate $p(\gamma)$
from $N_{\gamma}$, and the process is partially exchangeable.
\end{proof}
Partial exchangeability is interesting because of the following theorem
of Diaconis \& Freedman \cite{DF80}
\begin{thm}
Let $x=\{x_{0},x_{1},\dotsc\}$ be a random sequence of elements in
some state space, which is recurrent. Then $x$ is partially exchangeable
if and only if it is a mixture of Markov chains.
\end{thm}
Here recurrent means that the sequence returns to its starting vertex
with probability $1$. This has to be specified as for general random
sequences the various definitions of recurrence are not identical.
In fact, the counterexamples in \cite{DF80} which show that recurrence
is necessary consist of two states and a process that gets stuck in
the second state. A mixture of Markov chains means that there is some
measure $\mu$ on the space of Markov chains such that the probability
of every event $E$ is given by
\[
\int\mathbb{P}^{W}(E)\, d\mu(W)
\]
where $W$ is a Markov chain and $\mathbb{P}^{W}(E)$ is the probability
of $E$ under the Markov chain $W$. The same concept is known under
the name of random walk in random environment. Again, you first pick
an environment $W$ randomly, then perform random walk on it.
We already mentioned random walk in random environment when we discussed
the directed case above. But there each vertex was independent, while
in the general case this is typically not the case.

We will not prove the theorem, but just make some remarks on its proof.
Recurrence is used to break the infinite process into blocks, from
one visit of $x_{0}$ to the next. These blocks then turn out to be
\emph{fully} exchangeable, i.e.\ any permutation has the same probability.
This allows to apply de Finetti's theorem which states that such a
process is a mixture of i.i.d.\ processes. De Finetti's theorem itself
is not difficult either, it is an application of Krein-Milman. The
point to note here is that the proof is soft and abstract, without
any explicit calculations.

Let us demonstrate the utility of this fact by applying it to P\'olya's
urn, i.e.\ to LRRW on a line of length $3$ with $a=2$. There is
no question as to recurrence, since our walker returns to its starting
point every 2 steps without failure. Our random environment can only
relate to the central vertex, so asking ``what is the random environment?''
is equivalent to asking ``what is the probability to go left from
the central vertex?''. In other words, there is only one parameter.
Let us denote it by $p$.
\begin{thm}
\label{thm:polyaunif}The parameter $p$ defined above is uniform
in $[0,1]$.\end{thm}
\begin{proof}
Fix $p$ and examine the Markov chain on a line of length $3$ which
moves from the central vertex to the left vertex with probability
$p$ and to the right vertex with probability $1-p$. Denote by $L_{t}$
the number of times the left edge was crossed by time $t$. By the
law of large numbers we know that 
\[
\lim_{t\to\infty}\frac{L_{t}}{t}=p\qquad\mbox{almost surely.}
\]
The random walk in random environment tells us that for the measure
$\mu$ that we seek,
\[
\mathbb{P}\left(\frac{L_{t}}{t}\in E\right)=\int_{0}^{1}\mathbb{P}^{p}\left(\frac{L_{t}}{t}\in E\right)d\mu(p).
\]
This holds for every measurable $E$, but let us restrict our attention
to intervals. By the explicit solution of P\'olya's urn (theorem
\ref{thm:polya} above), the left hand side is $|E|+O(1/t)$ where
$|E|$ is the Lebesgue measure of $E$. Taking $t\to\infty$ (which
can be done inside the integral as $\mu$ is a probability measure
and the integrand is bounded) gives
\[
|E|=\int_{0}^{1}\mathbf{1}\{p\in E\}\, d\mu(p).
\]
Since this holds for all intervals $E$, we get that $\mu$ is uniform,
as needed.
\end{proof}
In other words, the theorem of Diaconis and Freedman allows us to
get from the ``static'' description of the P\'olya urn we had before,
a ``dynamic'' description which is fuller. We now know that P\'olya's
urn is equivalent, \emph{as a process}, to choosing some number $p\in[0,1]$
uniformly and then picking each time a black ball with probability
$p$ and a white ball with probability $1-p$.

In general graphs it proved to be a significant challenge to find
an explicit formula for the environment, and this is what we turn
to next.

\section{The magic formula}

Let us consider for a moment finite graphs. Even a triangle would
prove to be highly nontrivial! LRRW is recurrent on any finite graph
--- this is not difficult to show, but we will skip it here. The theorem
of Diaconis \& Freedman promises us that \emph{some }random environment
exists such that LRRW is equivalent to random walk in the environment.
But what is it? It turns out there is a formula for it, known fondly
as the ``magic formula''. The history of its discovery is complicated
and the reader would be best served by reading about it in detail
in \cite{MOR08}. It turns out that the environment is always \emph{reversible}
(a fact which does not follow from Diaconis \& Freedman). This means
that there is a (random) weight function $W:E\to[0,\infty)$ where
$E$ is the set of edges of $G$ such that the probability to go from
$x$ to $y$ is 
\[
\frac{W(x,y)}{\sum\limits _{z\sim x}W(x,z)}.
\]
It is common to denote the denominator by $W(x)$ and we shall do
so below. Reversibility is perhaps the most important property a Markov
chain may or may not possess. Most notably a reversible Markov
chain gives rise to a \emph{self-adjoint} operator (on $l^{2}$ with
the weight being $W$) which is amenably to spectral analysis. 

Now, $W$ is not unique since multiplying $W$ by a constant gives
rise to exactly the same Markov chain. So we need to normalize and
we normalize so that $\sum_{x\sim y}W(x,y)=1$. For completeness,
let us state the magic formula for non-constant initial weights, i.e.\ assume
for every edge $e$ there is some $a_{e}$ and the transition probabilities
at time $t$ are $a_{e}+N(e,t)$, normalized. The magic formula is
then:
\begin{thm}
Let $G=(V,E)$ be a finite graph. Examine LRRW on $G$, starting from
$x_{0}$ with initial weights $a_{e}$, and let $\mu$ be the corresponding
random environment. Then the density of $\mu$ at a given $W$ is
given by
\[
\frac{1}{Z}\cdot\frac{{\displaystyle \prod_{e\in E}W(e)^{a_{e}-1}}}{{\displaystyle W(x_{0})^{\frac{1}{2}a_{x_{0}}}\prod_{v\in V\setminus\{x_{0}\}}W(v)^{\frac{1}{2}(a_{v}+1)}}}\sqrt{\sum_{T}\prod_{e\in T}W(e)}
\]
where the sum inside the square root is a sum over all spanning trees
of $G$; where $Z$ is a normalization constant; and where $a_{v}=\sum_{e\ni v}a_{e}$.
\end{thm}
The most problematic term is the sum over trees, as it introduces
long-range dependencies. By the matrix-tree theorem this sum can be
written as a determinant, but that does not simplify the formula significantly. 

As an exercise, let us analyse P\'olya's urn again:
\begin{proof}
[Second proof of theorem \ref{thm:polyaunif}] We use the magic formula.
All $a_{e}=2$, $a_{x_{0}}=4$ and for $v\ne x_{0}$ (i.e.\ one of
the side vertices), $a_{v}=2$. Since the normalization is $\sum W(e)=1$,
we have $W(x_{0})=1$. Denote the weight of the left edge by $p$,
so for one of the side vertices $W(v)=p$ while for the other it is
$1-p$. All in all this gives
\[
\frac{1}{Z}\cdot\frac{p(1-p)}{1^{2}\cdot p^{3/2}(1-p)^{3/2}}\sqrt{p(1-p)}
\]
The expression inside the square root is so simple because a line
graph has exactly one spanning tree (the graph itself), so the sum
is only over one term. Everything cancels and we get that the density
is $\frac{1}{Z}$, which of course means that $Z=1$, proving the
theorem.
\end{proof}
Throughout the last decade the magic formula stood in the center of
research on LRRW. Even the case of the ladder was not so easy, and
was done in \cite{MR05}. The next result achieved was a tightness
result that allowed to show that LRRW is always equivalent to a random
walk in random environment, even in the transient case (which is not
covered by Diaconis \& Freedman) \cite{MR07}. The most impressive
result, also due to Merkl \& Rolles \cite{MR09}, was the proof that
a certain two-dimensional graph is recurrent. A conjecture that goes
back to the 80s \cite[page 1241]{P88} is that LRRW on $\mathbb{Z}^{2}$
is recurrent, at least for $a$ sufficiently small. Merkl \& Rolles
got quite close, but their proofs required to replace each edge by
a line of length 130 or more.

Around the same period a surprising connection to certain quantum
models appeared, a topic we will discuss next.

\section{The hyperbolic sigma model}

The hyperbolic sigma model has its roots in physics. As some of the
readers (and, frankly, the author) may not be familiar with the necessary
background, we will be short and vague. Hopefully the description
will still be useful to some. Our starting point is supersymmetry.

Supersymmetry started its life as a branch of particle physics that
postulated that each of the known particles has a partner (a ``superpartner'')
which is similar in most aspects, but the superpartner of a fermion
is a boson and vice versa. The existence of these particles leads
to cancellations which would solve several theoretical problems with
the current model. These superpartners were never found in particle
accelerators, but the mathematics of these cancellations turned out
to be very fruitful.

One potential application of supersymmetry to problems which have
nothing to do with bosons or fermions is the Wegner-Efetov approach
to random matrices and to certain disordered quantum systems \cite{E97}.
The titular hyperbolic sigma model captures some of the mathematical
difficulties of this approach, and, conjecturally, most of the phenomenology.
It has been studied by Disertori, Spencer \& Zirnbauer \cite{SZ04,DS10,DSZ10}.
The model, roughly, is as follows. We are looking at functions $\sigma$
from the vertices of our graph (say $\mathbb{Z}^{d}$) into hyperbolic
space. The action (or energy) $A(\sigma)$ is a sum over the edges
of some interaction \emph{which respects the symmetries of the hyperbolic
space}. Then each $\sigma$ is given probability density proportional
to $e^{-\beta A}$ where $\beta$ is a parameter which is analogous
to the initial weights $a$ in LRRW.

The approach taken in \cite{SZ04} to tackle the hyperbolic sigma
model is to make a change of variable in the hyperbolic plane (change
to the so-called horospherical coordinates) and then integrate some of the
variables. The result was a sigma model with values in $\mathbb{R}$
but with a much more complicated interaction term, which included
a determinant very similar to the one appearing in the magic formula.
This caused significant curiosity but for a few years it was not
clear whether the similarity is coincidental or not. This mystery
was solved in November 2011.

\section{The solution}

It was Sabot and Tarr\`es who managed to make the connection between
LRRW and the hyperbolic sigma model rigorous \cite{ST}. They investigated
a variation on LRRW called the vertex-reinforced jump process, defined
in continuous time. On the one hand the vertex-reinforced jump process
has a random walk in random environment representation with the density
of each environment \emph{identical} to the formulas that appear in
the hyperbolic sigma model. This allowed them to harness the results
of \cite{DS10} to show that the vertex-reinforced jump process is
recurrent for small $a$ in any graph (the jump process has a parameter
which acts analogously to $a$ for LRRW). On the other hand, they showed
that LRRW is given by a vertex-reinforced jump process with random
weights. Thus, after some modification of the techniques of \cite{DS10}
to work with different, unbounded $\beta$ in different edges, they
were able to show:
\begin{thm}
For every graph $G$ there exists some $a_{0}$ such that LRRW is
recurrent for all $a<a_{0}$.
\end{thm}
Simultaneously, with Angel and Crawford, we arrived at the same result
\cite{ACK}. Our approach was completely different and did not use
the magic formula at all. The basic idea is very simple.
\begin{proof}
[Basic idea of the proof]Examine some edge $e=(x,y)$ and assume
you have traversed $e$ for the first time from $x$ to $y$. Let
$f$ be the edge through which you arrived into $x$. If $a$ is extremely
small, then upon arriving to $x$ you see one edge, $f$, with weight
$1+a$, and all the others with weight $a$. Hence you are very likely
to return via $f$. This will repeat on the second, third etc.\ visit
to $x$. At each time, because $f$ has much larger weight than $e$
(its weight is even increasing, but we will not use that), it will
choose to exit $x$ through $f$ multiple times until finally exiting
through $e$. Denote the number of times it exited through $f$ before
the first exit through $e$ by $N$.

We now recall that the process is also a random walk in random environment.
The environment might be complicated, but we have earned a bit of
information about it. We know that the walker exited the vertex $x$
$N$ times through $f$ before the first time it exited through $e$.
This must mean that the weight of $f$ is approximately $N$ times
larger than that of $e$! In other words, we see that the weight of
edges decays exponentially in the ``first entrance distance'', i.e.\ if
we define $f_{2}$ as the edge through which we entered the other
vertex of $f$, $f_{3}$ as the edge through which we entered the
other vertex of $f_{2}$ etc., we finally reach $f_{n}=x_{0}$. Then
the argument says that the weight of $e$ should decay exponentially
in $n$.
\end{proof}
Comparing to the approach of Disertori-Spencer-Sabot-Tarr\`es, our approach
is much softer, and has almost no explicit calculations. We rely on
the existence of a random walk in random environment representation,
i.e.\ on the theorem of Diaconis \& Freedman, but that too, as explained
above, is a soft result.

Trying to formalize this leads to some minor technical difficulties.
The first is that, to get exponential decay, different edge ratios
need to be independent. However, different edges are not independent,
because the ``first entrance path'' is a global property which introduces
dependencies between all quantities. To tackle this difficulty we
use a somewhat brutish approach and simply count over all such paths,
i.e. write
\[
\mathbb{E}\left(W_{e}\right)=\sum_{\gamma}\mathbb{E}\left(W_{e}\cdot\mathbf{1}\{\gamma\mbox{ is the first entrance path}\}\right).
\]
This allows to just fix two edges $e$ and $f$ sharing a vertex $x$,
and define a variable as in the basic idea. These variables (call
them $R_{e,f}$) are independent and have expectation $\le C\sqrt{a}$.
Thus we argue, very roughly, as follows 
\begin{multline}
\mathbb{E}\left(W_{e}\cdot\mathbf{1}\{\gamma\mbox{ is the first entrance path}\}\right)=\\
\mathbb{E}\left(\prod_{i=1}^{\len\gamma-1}R_{(\gamma_{i-1},\gamma_{i}),(\gamma_{i},\gamma_{i+1})}\cdot\mathbf{1}\{\gamma\mbox{ is the first entrance path}\}\right)\le\\
\mathbb{E}\left(\prod_{i=1}^{\len\gamma-1}R_{(\gamma_{i-1},\gamma_{i}),(\gamma_{i},\gamma_{i+1})}\right)=\prod\mathbb{E}R_{(\gamma_{i-1},\gamma_{i}),(\gamma_{i},\gamma_{i+1})}\le\left(C\sqrt{a}\right)^{\len\gamma}.\label{eq:WR}
\end{multline}
The first equality is because under the event that $\gamma$ really
is the first entrance path, $W_{e}$ is really the product of the
$R_{e,f}$ along $\gamma$. The inequality that follows comes from
just throwing this event away. This seems very wasteful, but look
at the final result: we show that the expectation is exponentially
small, and not only that, we can make the exponent base as small as
we want by reducing $a$. So losing a fixed exponential term by counting
over the $\gamma$ is something that our approach can sustain. The
third step in (\ref{eq:WR}) is simply independence, and the fourth
is the estimate that came from the ``basic idea'' sketch above.

Another technical difficulty that might be worth mentioning is that,
even though the weights typically decrease along the first entrance
path, sometimes they increase. Once you incorporate this phenomenon
into the basic sketch, you see that the $R_{e,f}$, despite being
typically $C\sqrt{a}$, in fact have infinite expectation. We solve
this by taking everything to some small fixed power ($\frac{1}{4}$
worked). Thus (\ref{eq:WR}) as written is false, and what we actually
prove is 
\[
\mathbb{E}\left(W_{e}^{1/4}\right)\le(C\sqrt{a})^{\dist(e,x_{0})}.
\]
Returning to our favourite example, P\'olya's urn, we see that the
phenomenon is real: the ratio between the weights of the two edges
has no first moment. So some power must appear in the formulation
of the result, though $\frac{1}{4}$ is not optimal. With these two
problems solved, we were able to show the exponential decay of weights.
Let us state this formally:
\begin{thm}
Let $G$ be a graph with all degrees bounded by $K$. Then there exists
$a_{0}=a_{0}(K)>0$ such that for all $a\in(0,a_{0})$, 
\[
\mathbb{E}\left(W_{e}^{1/4}\right)\leq2K\Big(C(K)\sqrt{a}\Big)^{\dist(e,v_{0})}.
\]

\end{thm}

\section{Where now?}

The next step to understand the phase transition is to understand
the transient regime. Using our techniques we were able to show the
existence of a transient regime on any non-amenable graph. But getting
to $\mathbb{Z}^{d}$ seems to require a new idea. 

Next one would like to show that the phase transition is unique, namely
that there is some $a_{c}$ such where the process moves from recurrence
to transience, and not, say, that the process switches from recurrence
to transience and back a few times before finally becoming transient
for high enough $a$. In many models proving this kind of fact requires
monotonicity, but there are also other approaches, so e.g.\ \cite{H}.

There are also many other models we would like to know similar results.
We already covered other reinforced walks, but the connection to the
hyperbolic sigma models suggests many more questions. Are there other
sigma models where one can get an interacting representation for random
walk on the generated field? All-in-all this is an exciting and fast-moving
field, with a lot more to understand.


\begin{thebibliography}{10}
\bibitem{ACK}Omer Angel, Nicholas Crawford and Gady Kozma, \emph{Localization
for Linearly Edge Reinforced Random Walks}. Preprint (2012). Available
at: \href{http://arxiv.org/abs/1203.4010}{\nolinkurl{arXiv:1203.4010}}

\bibitem{AG11}Amine Asselah and Alexandre Gaudilliere, \emph{Sub-logarithmic
fluctuations for internal DLA}. Preprint (2011). Available at: \href{http://arxiv.org/abs/1011.4592}{\nolinkurl{arXiv:1011.4592}}

\bibitem{BDCKL}Itai Benjamini, Hugo Duminil-Copin, Gady Kozma and
Cyrille Lucas. In preparations.

\bibitem{DF80}Persi Diaconis and David A. Freedman, \emph{De Finetti's
Theorem for Markov Chains}. Ann. Probab. 8:1 (1980), 115--130. Available
at: \href{http://projecteuclid.org/euclid.aop/1176994828}{\nolinkurl{projecteuclid.org}}


\bibitem{DS10} Margherita Disertori and Thomas Spencer, \textit{Anderson
Localization for a Supersymmetric Sigma Model.} Commun. Math. Phys.
300:3 (2010), 659--671. Available at: \href{http://www.springerlink.com/content/m6583504661746lj/}{\nolinkurl{springerlink.com}},
\href{http://arxiv.org/abs/0910.3325}{\nolinkurl{arXiv:0910.3325}}

\bibitem{DSZ10} Margherita Disertori, Thomas Spencer and Martin R.
Zirnbauer, \textit{Quasi-diffusion in a 3D Supersymmetric Hyperbolic
Sigma Model.} Commun. Math. Phys. 300:2 (2010), 435--486. Available
at: \href{http://www.springerlink.com/content/61xj2m4860u40844/}{\nolinkurl{springerlink.com}}

\bibitem{DKL02}Rick Durrett, Harry Kesten and Vlada Limic, \emph{Once
edge-reinforced random walk on a tree}. Probab. Theory Related Fields
122:4 (2002), 567--592. Available at: \href{http://www.springerlink.com/content/xw5b9462m2kwtf6j/}{\nolinkurl{springerlink.com}}

\bibitem{E97}Konstantin Efetov, \textit{Supersymmetry in Disorder
and Chaos.} Cambridge Univ. Press., 1997.

\bibitem{ES03}Nathanaël Enriquez and Christophe Sabot, \emph{Edge
oriented reinforced random walks and RWRE}. C. R. Math. Acad. Sci.
Paris 335:11 (2002), 941--946. Available from: \href{http://www.sciencedirect.com/science/article/pii/S1631073X02025803}{\nolinkurl{sciencedirect.com}},
\href{http://www.proba.jussieu.fr/pageperso/enriquez/PUB/CRAS.pdf}{\nolinkurl{jussieu.fr/pageperso/enriquez}} 

\bibitem{H}Alan Hammond, \emph{Sharp phase transition in the random
stirring model on trees}. Preprint (2012). Available at: \href{http://arxiv.org/abs/1202.1322}{\nolinkurl{arXiv:1202.1322}}

\bibitem{JLS12}David Jerison, Lionel Levine and Scott Sheffield,
\emph{Logarithmic fluctuations for internal DLA}. J. Amer. Math. Soc.
25:1 (2012), 271--301. Available at: \href{http://dx.doi.org/10.1090/S0894-0347-2011-00716-9}{\nolinkurl{ams.org}},
\href{http://arxiv.org/abs/1010.2483}{\nolinkurl{arXiv:1010.2483}}

\bibitem{JLS11}David Jerison, Lionel Levine and Scott Sheffield,
\emph{Internal DLA and the Gaussian free field. }Preprint (2011).
Available at: \href{http://arxiv.org/abs/1101.0596}{\nolinkurl{arXiv:1101.0596}}

\bibitem{L95}Gregory Lawler, \emph{Subdiffusive fluctuations for
internal diffusion limited aggregation}. Ann. Probab. 23:1 (1995),
71--86. \href{http://projecteuclid.org/euclid.aop/1176988377}{\nolinkurl{projecteuclid.org}}

\bibitem{LBG92}Gregory Lawler, Maury Bramson and David Griffeath,
\emph{Internal diffusion limited aggregation}. Ann. Probab. 20:4 (1992),
2117--2140. Available at: \href{http://projecteuclid.org/euclid.aop/1176989542}{\nolinkurl{projecteuclid.org}}

\bibitem{L03}Vlada Limic, \emph{Attracting edge property for a class
of reinforced random walks}. Ann. Probab. 31:3 (2003), 1615--1654.
Available at: \href{http://projecteuclid.org/euclid.aop/1055425792}{\nolinkurl{projecteuclid.org}}

\bibitem{LT07}Vlada Limic and Pierre Tarr\`es, \emph{Attracting
edge and strongly edge reinforced walks}. Ann. Probab. 35:5 (2007),
1783--1806. Available at: \href{http://projecteuclid.org/euclid.aop/1189000928}{\nolinkurl{projecteuclid.org}}

\bibitem{LT08}Vlada Limic and Pierre Tarr\`es, \emph{What is the difference
between a square and a triangle?} In and out of equilibrium. 2, 481--495,
Progr. Probab., 60, Birkhäuser, Basel, 2008. Available at: \href{http://arxiv.org/abs/0712.0958}{\nolinkurl{arXiv:0712.0958}}

\bibitem{MOR08}Franz Merkl, Aniko \"{O}ry and Silke W. W. Rolles,
\emph{The ``Magic Formula'' for Linearly
Edge-reinforced Random Walks}. Statistica Neerlandica 62:3 (2008),
345--363. Available at: \href{http://onlinelibrary.wiley.com/doi/10.1111/j.1467-9574.2008.00402.x/abstract}{\nolinkurl{wiley.com}},
\href{http://www-m5.ma.tum.de/foswiki/pub/M5/Allgemeines/SilkeRollesPublications/magic-snversion.pdf}{\nolinkurl{tum.de}}

\bibitem{MR05}Franz Merkl and Silke W. W. Rolles, \emph{Edge-reinforced
random walk on a ladder}. Ann. Probab. 33:6 (2005), 2051--2093. Available
at: \href{http://projecteuclid.org/euclid.aop/1133965853}{\nolinkurl{projecteuclid.org}}

\bibitem{MR07}Franz Merkl and Silke W. W. Rolles, \emph{Asymptotic
Behavior of Edge-reinforced Random Walks}. Ann. Probab. 35:1 (2007),
115--140. Available at: \href{http://projecteuclid.org/euclid.aop/1174324125}{\nolinkurl{projecteuclid.org}}

\bibitem{MR09}Franz Merkl and Silke W. W. Rolles, \emph{Recurrence
of edge-reinforced random walk on a two-dimensional graph}. Ann. Probab.
37:5 (2009), 1679--1714. Available at: \href{http://projecteuclid.org/euclid.aop/1253539854}{\nolinkurl{projecteuclid.org}},
\href{http://arxiv.org/abs/math/0703027}{\nolinkurl{arXiv:math/0703027}}

\bibitem{P88}Robin Pemantle, \emph{Phase Transition in Reinforced
Random Walk and RWRE on Trees}. Ann. Probab. 16:3 (1988), 1229--1241.
Available at: \href{http://projecteuclid.org/euclid.aop/1176991687}{\nolinkurl{projecteuclid.org}}

\bibitem{P30}George P\'olya, \emph{Sur quelques points de la th\'eorie
des probabilit\'es} {[}French: On some points of the theory of probability{]}.
Ann. inst. Henri Poincaré, 1:2 (1930), 117--161. Available at: \href{http://www.numdam.org/numdam-bin/fitem?id=AIHP_1930__1_2_117_0}{\nolinkurl{numdam.org}}

\bibitem{S11}Christophe Sabot, \emph{Random walks in random Dirichlet
environment are transient in dimension $d\geq3$}. Probab. Theory
Related Fields 151:1-2 (2011), 297--317. \href{http://www.springerlink.com/content/24480204un6v607x/}{\nolinkurl{springerlink.com}},
\href{http://arxiv.org/abs/0811.4285}{\nolinkurl{arXiv:0811.4285}}

\bibitem{ST}Christophe Sabot and Pierre Tarr\`es, \textit{Edge-reinforced
random walk, Vertex-Reinforced Jump Process and the supersymmetric
hyperbolic sigma model.} Preprint (2012). Available at: \href{http://arxiv.org/abs/1111.3991}{\nolinkurl{arXiv:1111.3991}}

\bibitem{SZ04}Thomas Spencer and Martin R. Zirnbauer, \emph{Spontaneous
symmetry breaking of a hyperbolic sigma model in three dimensions}.
Commun. Math. Phys. 252:1-3 (2004), 167--187. Available at: \href{http://www.springerlink.com/content/06tcd263nyxkxyxu/}{\nolinkurl{springerlink.com}},
\href{http://arxiv.org/abs/math-ph/0410032}{\nolinkurl{arXiv:math-ph/0410032}}\end{thebibliography}
\end{document}